\documentclass[11pt]{amsart}
\usepackage{mathrsfs}
\usepackage{amsmath}
\usepackage{cases}
\usepackage{amssymb}
\usepackage{graphicx}
\usepackage{color}
\usepackage{booktabs}
\usepackage{bigstrut}
\usepackage{CJK}

\begin{document}
\begin{CJK*}{GBK}{song}
\newtheorem{theorem}{Theorem}[section]
\newtheorem{prop}[theorem]{Proposition}
\newtheorem{lemma}[theorem]{Lemma}
\newtheorem{corollary}[theorem]{Corollary}
\newtheorem{defn}[theorem]{Definition}
\newtheorem{conjecture}[theorem]{Conjecture}
\newtheorem{remark}[theorem]{Remark}
\newtheorem{exe}{Exercise}

\theoremstyle{definition}%
\newtheorem{definition}{Definition}

\def\pasdegrille{\let\grille = \pasgrille}
\def\ecriture#1#2{\setbox1=\hbox{#1}
\dimen1= \wd1 \dimen2=\ht1 \dimen3=\dp1 \grille #2 \box1 }
\def\aat#1#2#3{
\divide \dimen1 by 48 \dimen3=\dimen1 \multiply \dimen1 by #1
\advance \dimen1 by -\dimen3 \divide \dimen1 by 101 \multiply
\dimen1 by 100 \divide \dimen2 by \count11 \multiply \dimen2 by #2
\setbox0=\hbox{#3}\ht0=0pt\dp0=0pt
  \rlap{\kern\dimen1 \vbox to0pt{\kern-\dimen2\box0\vss}}\dimen1= \wd1
\dimen2=\ht1}
\def\pasgrille{
\count12= \dimen1 \divide \count12 by 50 \divide \dimen2 by \count12
\count11 =\dimen2 \ \divide \dimen1 by 48
\setlength{\unitlength}{\dimen1} \smash{\rlap{\ }} \dimen1= \wd1
\dimen2=\ht1 }
\def\grille{
\count12= \dimen1 \divide \count12 by 50 \divide \dimen2 by \count12
\count11 =\dimen2 \ \divide \dimen1 by 48
\setlength{\unitlength}{\dimen1}
\smash{\rlap{\graphpaper[1](0,0)(50, \count11)}} \dimen1= \wd1
\dimen2=\ht1 }

\pasdegrille

\numberwithin{equation}{section}

\newcommand{\R}{\mathbb R}
\newcommand{\TT}{\mathbb T}
\newcommand{\Z}{\mathbb Z}
\newcommand{\N}{\mathbb N}
\newcommand{\Q}{\mathbb Q}
\newcommand{\Sol}{\operatorname{Sol}}
\newcommand{\ND}{\operatorname{ND}}
\newcommand{\ord}{\operatorname{ord}}
\newcommand{\leg}[2]{\left( \frac{#1}{#2} \right)}  
\newcommand{\Sym}{\operatorname{Sym}}
\newcommand{\vE}{\mathcal E} 
\newcommand{\ave}[1]{\left\langle#1\right\rangle} 
\newcommand{\Var}{\operatorname{Var}}
 \newcommand{\tr}{\operatorname{tr}}
\newcommand{\supp}{\operatorname{Supp}}
\newcommand{\intinf}{\int_{-\infty}^\infty}
\newcommand{\beq}{\begin{equation}}
\newcommand{\eeq}{\end{equation}}
\newcommand{\ben}{\begin{eqnarray}}
\newcommand{\een}{\end{eqnarray}}
\newcommand{\beno}{\begin{eqnarray*}}
\newcommand{\eeno}{\end{eqnarray*}}
\newcommand{\dist}{\operatorname{dist}}
\newcommand{\area}{\operatorname{area}}
\newcommand{\vol}{\operatorname{vol}}
\newcommand{\diam}{\operatorname{diam}}
\newcommand{\Sp}{\mathbb S}

\newcommand{\Bin}{\operatorname{Binom}}
\newcommand{\pois}{\operatorname{Pois}}
\newcommand{\myK}{V}
\newcommand{ \myvar}{\operatorname{Var}}
\newcommand{\ripleyK}{\^K}
\newcommand{\E}{\mathbb{E}}
\newcommand{\T}{\mathbb{T}}
\newcommand{\Prob}{\operatorname{Prob}}
\newcommand{\rta}{\rightarrow}
\def\RR{{\rm \textbf{R}}}
\def\TT{{\rm \textbf{T}}}
\def\SS{{\rm \textbf{S}}}
\def\ZZ{{\rm \textbf{Z}}}
\newcommand{\q}{\quad}
\def\d{\delta}
\def\a{\alpha}
\def\e{\varepsilon}
\def\ld{\lambda}
\def\p{\partial}
\def\v{\varphi}
\def\D{\Delta}
\renewcommand{\^}{\widehat}
\newcommand{\bx}{\mathbf x}
\newcommand{\by}{\mathbf y}
\newcommand{\bz}{\mathbf z}
\newcommand{\bl}{\mathbf \lambda}
\newcommand{\SF}{\mathcal H}
\newcommand{\V}{\vskip0.2cm}

\title[Maximal spherical means]{On local smoothing problems and Stein's maximal spherical means}
\author[Miao et al]{Changxing Miao, Jianwei Yang and Jiqiang Zheng}

\date{\today}

\address{Institute of Applied Physics and Computational Mathematics, Beijing 100088, China}
\email{miao\_changxing@aliyun.com ,  miao\_{}changxing@iapcm.ac.cn}

\address{LAGA(UMR 7539), Institut Galil\'ee, Universit\'e Paris 13, Sorbonne Paris Cit\'e, France\\
Beijing International Center for Mathematical Research, Peking University, Beijing 100871, China} \email{geewey\_{}young@pku.edu.cn}

\address{Universit\'e de Nice - Sophia Antipolis,
Laboratoire J. A. Dieudonn\'e, 06108 Nice Cedex 02, France}
\email{zhengjiqiang@gmail.com}


\begin{abstract}
It is proved that the local smoothing conjecture for wave equations implies certain
improvements on Stein's analytic family of maximal spherical means.
Some related problems are also discussed.
\end{abstract}

\maketitle

\section{Introduction}\label{sect:introd}
For $\a>0$, we let
$m_\a(x)=\Gamma(\a)^{-1}(1-|x|^2)^{\a-1}_+$
where $x\in \R^n$, $\Gamma(\a)$ is the Gamma function and
$r_+$ is a homogeneous distribution defined to be $r$ when $r>0$, and equal to $0$ if $r\leq 0$.
Denote by $m_{\a,\,t}(x)=m_\a(x/t)\,t^{-n}$ for
$t>0$ and define
\beq\label{eq:def-max-sph}
 \mathscr{M}^\a_tf(x)=\bigl(f*m_{\a,\,t}\bigr)(x),
\eeq
initially for $f\in C^\infty_0(\R^n)$.\V

These averaging operators are defined
{\em a priori} only for real positive $\a$.
However, if we recall the Fourier transform of $m_\a$
\beq\label{eq:FT-of-m}
 \^m_{\a}(\xi)
 =\pi^{-\a+1}|\xi|^{-\frac n2-\a+1}
 \mathcal J_{\frac n2+\a-1}(2\pi|\xi|),
\eeq
where
$\mathcal J_m(x)$
is the Bessel function of order $m$
( see \cite{Stein-Weiss1971} or Appendix), and
notice that ${\rm Re}\,\a\leq 0$ is allowed in the Bessel functions in \eqref{eq:FT-of-m}
by means of analytic continuation, we can extend the notion of \eqref{eq:def-max-sph} to include complex $\a$ via Fourier transform
\beq\label{eq:def-max-sph-ft}
  \^{\mathscr{M}^\a_tf}(\xi)
 =\^m_\a(\xi t)\^f(\xi),\;
  f\in C^\infty_0(\R^n).
\eeq
It is also important to remark that $\^m_\a(0)$ is finite and $\^m_\a(\xi)$ is smooth near the origin.
For details on these standard facts, we refer to \cite{Stein1976} and Theorem 4.15, Chap IV in \cite{Stein-Weiss1971}.\V

It is not hard to see that in the sense of distribution
$$
  \lim_{\a\rta\,0+ } \frac{1}{\Gamma(\a)}\,
  t^{\a-1}_+=\d(t),
$$
where $\d(t)$ denotes the Dirac distribution at zero. In particular,
for $f\in C^\infty_0(\R^n)$, we have by the co-area formula
\beq\label{sph-0}
  \mathscr M^0_tf(x)
  =\lim_{\a\,\rta0+}\mathscr M^\a_tf(x)
  =c_n \int_{S^{n-1}}f(x+yt)\,
  d\sigma(y),
\eeq
where $c_n$ is a constant
depending only on $n$,
$S^{n-1}$ denotes the standard unit sphere in
$\R^n$ and $d\sigma$
corresponds to the normalized surface measure
induced from Lebesgue measure on $\R^n$.
In what follows, we call
\eqref{sph-0} the spherical means of $f$.\V

Let $M^\a$ be the maximal operator associated to $\mathscr M^\a_t$ defined as
\beq\label{eq:def-max-spherical}
  M^\a (f)(x)=\sup_{t>0}\bigl|\mathscr M^\a_tf(x)\bigr|\,.
\eeq
In \cite{Stein1976}, Stein proved when $n\geq 3$, one has
\beq\label{eq:stein}
  \bigl\|M^\a (f)\bigr\|_{L^p(\R^n)}\leq A_{p,\,\a}\|f\|_{L^p(\R^n)},\;
  0<A_{p,\,\a}<+\infty\,,
\eeq
under the following condition
\beq
\label{eq:stein-1}
{\rm Re}\;\a>1-n+\frac np, \quad \text{if}  \quad 1<p\leq2\,,
\eeq
or
\beq
\label{eq:stein-2}
{\rm Re}\;\a>\frac{2-n}{p}, \quad \text{if} \quad 2\leq p\leq \infty\,,
\eeq
where the two dimensional case was left open since then the
necessary condition $p>2$ for $\a=0$ eliminates the use of $L^2$ argument
based on Plancherel's theorem.
The above two \emph{admissible} relations  for $\a$ and $p$ are summarized when $n\geq 3$
in Figure 1, where the relation \eqref{eq:stein-2} corresponds to the dotted segment $OB$.
\begin{figure}[ht]
\begin{center}
$$\ecriture{\includegraphics[width=4cm]{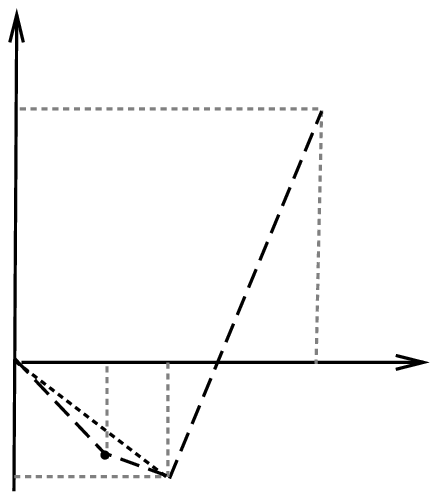}}
{\aat{-1}{55}{${\rm Re}\,\a$}
\aat{-1}{45}{$1$}
\aat{-6}{4}{$\frac {2-n}2$}
\aat{19}{20}{$\frac{1}{2}$}
\aat{6}{20}{$\frac{n-1}{2(n+1)}$}
\aat{-1}{16}{$O$}
\aat{7}{5}{$A$}
\aat{20}{0}{$B$}
\aat{36}{10}{$1$}
\aat{43}{10}{$\frac1p$}}$$
\end{center}
\caption{}
\end{figure}
\V
The condition \eqref{eq:stein-2} is not optimal. An important problem
is how to extend this range for $\a$ and $p$, where $2< p\leq \infty$.
In this paper, we are interested in investigating on this problem and show that it is possible to get certain improvement by using the local smoothing
estimate of linear wave equations. More precisely, our main theorem reads
\begin{theorem}\label{thm:main}
Suppose $n\geq 2$ and for $p\geq 2$ denote by
\beq\label{eq:p-n}
\varepsilon(p,n)=\min\left\{\frac{n-1}{4}+\frac{n-3}{2p},\frac{n-1}{p}\right\}.
\eeq
Then, \eqref{eq:stein}
is valid for all $\a\in \mathbb C$  and $p>2$ such that
${\rm Re}\,\a>-\varepsilon(p,n)$.
\end{theorem}
\begin{remark}
\label{rmk:improve}
Compared to \eqref{eq:stein-2}, there is a $1/p-$downwards extension for the range of ${\rm Re}\, \a$
so that \eqref{eq:stein} is valid for $p\geq p_{n}$ where
$\displaystyle p_{n}=2\frac{n+1}{n-1}$, and there is a $\displaystyle \frac{n-1}{2}\left(\frac12-\frac1p\right)-$improvement for $2<p<p_{n}$.
This is indicated in Figure 1 by the dashed line segments $OA$ and $AB$.
\end{remark}

The proof of Theorem \ref{thm:main} relies on
the recent results concerning Sogge's
local smoothing conjecture for wave equations.
Let $u(t,x)$
be the solution to the Cauchy problem
\begin{equation} \label{eq:1}
\left\{ \begin{aligned}
         (\partial^2_t-\Delta )& u(x,t)= 0,\,(x,t)\in\R^n\times\R\,,\\
                 u(x,0) &=f(x)\,,\;
                  \partial_tu(x,0)=g(x)\,.\\
                          \end{aligned} \right.
                          \end{equation}
It is conjectured in \cite{Sogge1991}
that for $n\geq 2$ and $\displaystyle p\geq\frac{2n}{n-1}$, one has
\beq\label{LSC}
  \|u\|_{L^p(\R^n\times [1,2])}\leq C\bigl(\|f\|_{W^{\gamma,\,p}(\R^n)}+\|g\|_{W^{\gamma-1,\,p}(\R^n)}\bigr)\,,\,
  \gamma>\frac{n-1}2-\frac np\,,
\eeq
where $W^{\gamma,\,p}$ represents the usual inhomogeneous Sobolev space.
Indeed, the reader will realize that
what we proved below
is nothing but the fact that \eqref{LSC} implies the above downwards extension of the range of $\a $,
and Theorem \ref{thm:main} follows simply
from the best knowledge on $p$ for which
\eqref{LSC} is true.
\V

Some pioneer results on \eqref{LSC}
with loss of derivatives appeared chronologically in\,\cite{Sogge1991,MoSSI,MoSSII,Sogge1993}.
Their arguments involve mainly, among other things, orthogonality, square-function estimates and certain
variable coefficient versions of Kakeya-Nykodim type maximal inequalities, as can be found in \cite{Sogge1993}
and  Chapter X in \cite{Stein1993}.
\V

In \cite{Wolff2000b}, by proving a sharp decoupling
inequality, Wolff first obtained  \eqref{LSC} in dimension two
for all $ p>74$.
The higher dimensional counterpart with $n\geq 3$ was
established in \cite{LabaWolff2002} for
$\displaystyle p>\min\left\{2+\frac8{n-3}\,,\,2+\frac{32}{3n-7}\right\}$, which was improved later
in \cite{Garrigos-Schlag-Seeger} to
$\displaystyle p>2+\frac8{n-2}\cdot\frac{2n+1}{2(n+1)}$ .
The best result is obtained in a very recent work \cite{BourgainDemeter},  where
Bourgain and Demeter proved \eqref{LSC} is true for
\[p\geq \frac{2(n+1)}{n-1}\,,\; n\geq 2.\]
Their argument is based on the techniques
developed in a series of works on the decoupling inequalities. We refer to \cite{BourgainDemeter}
for more references and comments on this issue.
\V

If we use interpolation, we have
\beq\label{LSC-2}
  \|u\|_{L^p(\R^n\times [1,2])}\leq C\|f\|_{W^{\gamma-1,\,p}(\R^n)}\,,
\eeq
for $2<p\leq\infty$ and $\gamma>\gamma(p,n)$, where
\begin{equation}
\label{range}
 \gamma(p,n)=\max\left\{\frac{n-1}{2}\left(\frac{1}2-\frac1p\right)\,,\,\frac{n-1}{2}-\frac{n}{p}\right\}\,.
 \end{equation}
It is \eqref{LSC-2} and \eqref{range} that corresponds to the
improvement upon the relation of ${\rm Re}\,\a$ and $p$
obtained in Theorem \ref{thm:main}.

\begin{remark}
When $p<\frac{2n}{n-1}$, there is no additional $1/p$ local smoothing for wave equations,
nor for more general Fourier integral operators satisfying cinematic curvature conditions, as pointed out in \cite{Sogge1993}.
\end{remark}

Let us turn back to the spherical means and give
some historical remarks.
In the two dimensional case when $\a=0$, the $L^p-$boundedness
of the circular maximal operator \eqref{eq:def-max-spherical}
for $p>2$ was, ten years around after \cite{Stein1976},
finally established by Bourgain \cite{Bourgain1985}.
Later, this result concerning circular maximal means was generalized to planar convex curves in \cite{Bourgain1986c}.
This result was extended to the variable coefficient version by Sogge in \cite{Sogge1991},
where it is shown that the translation invariance property of the curves are not essential.
However, to ensure a ``dynamic" condition, an assumption on {\em cinematic curvature} condition is required.
See \cite{Sogge1991,Sogge1993} and \cite{Stein1993} for more details.
Further more, results of this kind are extended to some possible $L^q-L^p$ estimate for certain $q>p$
by Schlag \cite{Schlag1997} using Kolasa-Wolff's geometric/combinatorial method \cite{Kolasa-Wolff1999}.
This result was recovered and generalized to the higher dimensional counterparts as well as to its variable coefficient cases
by Schlag-Sogge in \cite{Schlag-Sogge1997} using partial local smoothing estimates.\V

If one removes the restriction on $\a= 0$ and allows $\a$ to take complex values,
\eqref{eq:stein} was strengthened when $n=2$ by  Mockenhaupt, Seeger and Sogge
in \cite{MoSSI} for all ${\rm Re}\,\a>-\e(p)/2$, which extends \eqref{eq:stein-2} for
\begin{equation} \label{eq:11}
\e(p)=\left\{ \begin{aligned}
          &\frac12-\frac1p\;, \quad2<p<4\,; \\
           &\frac1p\; \quad\quad,\quad 4\leq p<\infty.
                          \end{aligned} \right.
                          \end{equation}
As far as we know, this is
the first work connecting the maximal circular means with the local smoothing
problems of wave equations, or more generally, Fourier integral operators satisfying
Sogge's {\em  cinematic curvature} condition.\V

Around the maximal operator \eqref{eq:def-max-spherical}, there are two possible ways of doing extensions.
One is to consider the variable coefficient version for \eqref{eq:stein} with $\a=0$.
These results are investigated
in \cite{Sogge1993}, including both $n\geq 3$ and $n=2$, where the two cases are treated
separately by means of Fourier integral operators.
Another direction is to extend the admissible relation for $\a$ and $p$ in \eqref{eq:stein-2} for
\eqref{eq:stein} as far as possible.
Thus, it follows a natural question whether one may
combine the two issues together by
extending \eqref{eq:stein} to the non-translation-invariance setting with $\alpha\neq 0$.\V

Before ending up this section, let us state
an application of Theorem \ref{thm:main} to
the wave equations.

\begin{corollary}
Suppose $n\geq 4$ and let $u(x,t)$ be the solution of the problem
\eqref{eq:1} with $f(x)\equiv 0$.
If  $\displaystyle p\in\left(\frac{2n}{n+1},\frac{2(n-1)}{n-3}\right)$ and $g\in L^{p}(\mathbb R^{n})$, then we have
\beq\label{eq:p-w}
  \lim_{t\rta 0}\frac{u(x,t)}{t}=g(x)\,,
\eeq
for almost every $x\in\R^n$.
\end{corollary}

\begin{proof}
As in \cite{Stein1976}, if we take $\a=\frac{3-n}2$ and $c_n=\frac12\pi^{-\frac n2-\frac12}$, then
\beq\label{eq:sph-wave}
  u(x,t)=c_n t\mathscr M^\a_t(g)(x),
\eeq
solves the Cauchy problem of the wave equation \eqref{eq:1} with $f(x)\equiv0$.
We refer to a straightforward interpretation on \eqref{eq:sph-wave}  in Appendix.
As a consequence of \eqref{eq:stein-1} and Theorem \ref{thm:main}, \eqref{eq:stein} is true for
\[ p\in \Bigl(\frac{2n}{n+1},\,2\,\Bigr]\,\bigcup\,\Bigl[\,\frac{2(n+1)}{n-1},\,\frac{2(n-1)}{n-3}\Bigr)\,.\]
Hence, we have \eqref{eq:stein}  for $\displaystyle p\in \Bigl(\frac{2n}{n+1},\frac{2(n-1)}{n-3}\Bigr)$ by interpolation.
From this, we conclude \eqref{eq:p-w} for the same range of $p$.
\end{proof}

\begin{remark}
Notice that the above
almost everywhere convergence result \eqref{eq:p-w} was proved for
$g\in L^p(\R^n)$, where  $\displaystyle \frac{2n}{n+1}<p<\frac{2(n-2)}{(n-3)}$ for $n\geq 3$ in in \cite{Stein1976}. This above corollary slightly improves this result when $n\geq 4$.
\end{remark}

To end up this section, we indicate that it is commented in \cite{Stein1993} that the optimal results for
 $p> 2$ and $n\geq 2$ ``are still a mystery".
Although we can show that under the assumption
of sharp local smoothing estimate for wave equations,
the admissible range as for \eqref{eq:stein} is enlarged as in Remark \ref{rmk:improve},
we do not know whether this is already optimal or not.\V

This paper is organized as follows.
In Section 2, we prove Theorem \ref{thm:main},
where the proof is divided into four steps.
Section 3 is devoted to some remarks and comments for further study around this topic.
In Appendix, we clarify certain identities used in the introduction.
Although these are rather standard facts,
we include them for the convenience of reading.\V

\textbf{Acknowledgments:}
The authors thank the referees
 for spending their time reading and comments which improve this paper a lot.
This work was supported in part by the National Natural Science Foundation of China under grant No.11231006, and  No.11671047.   C. Miao was also supported  by Beijing Center for Mathematics and Information
Interdisciplinary Sciences.
J. Yang was supported by ERC
Advanced Grant No. 291214 BLOWDISOL .
J. Zheng was partly supported by the European
Research Council, ERC-2012-ADG, project number
320845 : Semi-Classical Analysis of Partial
Differential Equations.

\section{Proof of Theorem \ref{thm:main}}\label{sect:proof}

This section is devoted to the proof of our main theorem.
We start with an outline of the argument.
First, we prove a truncated maximal function where the supremum is taken over $t\in [1,2]$.
To dominate the supremum, we use Sobolev embedding
 $W^{\beta,p}(I)\hookrightarrow L^\infty([1,2])$
where $I$ is a suitable compact interval containing $[1,2]$ and $\beta>\frac1p$.
This enforces us to define the fractional order derivative of $\mathscr M^\a_t(f)(x)$ in $t-$variable.
For this, we will, after applying Littlewood-Paley decomposition to $f$, write
for each $j\geq 1$
$$|\p_t|^\beta\mathscr {M}^\a_t (\D_jf)(x),$$
as an integration operator, where the distributional kernel
can be represented by the difference of the two following oscillatory integrals
\begin{align*}
\mathcal I_j(x,y)=&\iint  a_{1,j}(\tau,s,\xi)e^{2\pi i\phi_1(x,y,t;,\xi,\tau,s,\theta)}d\theta ds d\tau d\xi,\\
\mathcal J_j(x,y)=&\iint a_{2,j}(\tau,s,\xi)e^{2\pi i\phi_2(x,y,t;,\xi,\tau,s,r)}dr ds d\tau d\xi,
\end{align*}
where $ a_{1,j}(\tau,s,\xi)$ and $a_{2,j}(\tau,s,\xi)$ are two  appropriate symbols and the two  phase functions read
\begin{align*}
\phi_1(x,y,t;,\xi,\tau,s,\theta)=&(x-y)\cdot\xi+(t-s)\tau+s|\xi|\sin\theta-\theta\varpi/2\pi,\\
\phi_2(x,y,t;,\xi,\tau,s,r)=&(x-y)\cdot\xi+(t-s)\tau +i(|\xi|s\sinh r+\varpi r/2\pi).
\end{align*}
Observe that $\phi_1$ has critical points in the $s-$variable only if
$\tau\approx |\xi|\sin\theta$.
This suggests us to localize the
frequency of time by means of truncating $\tau$ to the low
frequency. Combined with Littlewood-Pelay decomposition and scaling,
the temporal regularity is transferred to the spacial derivatives.
In the proof, method of stationary phase and Schl\"afli's
integral representation of Bessel functions  in \cite{Wat} : for $r\in
\R^+$ and $k>-\frac12$,
\begin{equation}
\begin{split}
\mathcal J_k( r)=&\frac1{2\pi}\int^{\pi}_{-\pi}e^{ ir\sin
\theta}e^{-i\theta k}d\theta-\frac{\sin(k\pi)}{\pi}\int_0^\infty
e^{-(r\sinh(s)+ks)}ds\\
:=&\tilde J_k(r)-E_k(r)
\end{split}
\end{equation}
 play a central role. The former eliminates the error terms in the following arguments while the latter
gives rise to the half wave operator in Step 3 so that we may involve the sharp local smoothing estimate.
From now on, we always assume ${\rm Re}\,\a<0$ since this is the interesting situation.
At the end of the proof, we will eliminate the restriction on $t\in[1,2]$ by a standard trick.
Now, let us turn to the rigorous proof.
\begin{proof}[{\bf The proof of Theorem \ref{thm:main}}]
We take a function
$\varphi\in C^\infty_0(\R)$ such that ${\rm supp}\,\varphi\subset[1/2,2]$
to form a partition of unity $\sum\varphi(2^{-j}s)=1$,
where $ s\in \R\setminus\{ 0\}$ and the summation is taken over all integers $j\in \Z$.
For $\xi\in \R^n$, we define
$$\varphi_j(\xi)=\varphi(2^{-j}|\xi|),\;\text{ for }j\geq1,$$
$$\varphi_0(\xi)=1-\sum^\infty_{j=1}\varphi_j(\xi),$$ and write
$\^{\D_j f}(\xi)=\varphi_j(\xi)\^f(\xi)$ for $j\geq 0$,
where $\D_j$ denotes the well-known Littlewood-Pelay's projector. 
To use the reproducing property, we will also use $\tilde \varphi$
to represent a smooth positive function identical to one on the support of $\varphi$ and
vanishing outside the interval $(1/4,4)$.
We then define $\tilde \D_j f(x)$ via $\^{\tilde \D_j f}(\xi)=\tilde\varphi(2^{-j}|\xi|)\^f(\xi)$.
Notice that in the following argument, we may assume $\a$ is real and $\a<0$
whereas the general case for complex $\a$ follows the same reasoning.
The proof is divided into four steps.\V


\textbf{Step 1.}
In this step, we show  the low frequency part is well behaved,
namely there is some constant $C>0$ such that
\beq\label{eq:low-freq}
 \Bigl\| \sup_{1<t<2}|\mathscr M^\a_t (\D_0f)(x)|\Bigr\|_{L^p(\R^n)}\leq C \|f\|_{L^p(\R^n)}.
\eeq
In fact, it is immediate once we have
$$
\sup_{1<t<2}|\mathscr M^\a_t (\D_0f)(x)|\leq C M_{HL}(f)(x),
$$
where $M_{HL}$ denotes the standard Hardy-Littlewood maximal function.
This is because $\^m_\a(\xi t)\varphi_0(\xi)$ is smooth and supported in $|\xi|\leq 2$.
Thus \eqref{eq:low-freq} follows immediately from the $L^p$ boundedness of $M_{HL}$ for any $p>1$.\V

\textbf{Step 2.} In this step, we single out the main contribution of
$\mathscr M^\a_t(\D_jf)(x)$ for each $j$.
Choose a smooth positive function $\chi$ identical to one in a neighborhood of $[1,2]$ and
vanishing outside $(1/2,4)$.
Taking also a smooth positive function $\chi_0(\tau)$, identical to one on the interval $[-4,4]$
and zero outside $(-8,8)$, we set $\chi_{I_j}(\tau)=\chi_0(2^{-j}\tau)$.
We consider the space-time Fourier transform of $\chi(t)\mathscr {M}^\a_t (\D_jf)(x)$ as calculated below
\begin{align*}
 &\mathcal F_{(x,t)\rta(\xi,\tau)}\Bigl(\chi(t)\mathscr {M}^\a_t (\D_jf)(\cdot)\Bigr)(\xi,\tau)\\
  =&|\xi|^{-\varpi}\varphi_j(\xi)\^f(\xi)\biggl[\int_{-\pi}^{\pi}
  \^{\chi_1}(\tau-|\xi|\sin\theta)e^{-i\theta\varpi}d\theta
  -\int^\infty_0\^{\chi_2}\bigl(\tau-i|\xi|\sinh s\bigr)e^{-\varpi s}ds\biggr],
\end{align*}
where $\varpi=\frac n2+\a-1$, and
$$\chi_1(t)=\frac1{2\pi^\a}\chi(t)t^{-\varpi},\; \chi_2(t)=\chi(t)t^{-\varpi}\frac{\sin\varpi \pi}{\pi^\a}.$$
We note that $\^\chi_2(\tau)$ can be extended to the upper half complex plane
$\{\tau\in\mathbb C:{\rm Im}\,\tau\geq0\}$ as an analytic function since $\chi_2$ is smooth and compactly supported.\V

Thus, we may define for $j\geq 1$,
\beq\label{eq:frac-t}
 |\p_t|^\beta\Bigl(\chi(t)\mathscr {M}^\a_t (\D_jf)(x)\Bigr)
 =\iint e^{2\pi i(x\cdot\xi+t\tau)}a_j(\tau,\xi)\^f(\xi)d\tau d\xi,
\eeq
and
\beq\label{eq:F-j}
 \mathscr F_j f(x,t)
 =\iint e^{2\pi i(x\cdot\xi+t\tau)}\chi_{I_j}(\tau) a_j(\tau,\xi)\^f(\xi)d\tau d\xi,
\eeq
with an amplitude $a_j(\tau,\xi)=a_{j,1}(\tau,\xi)-a_{j,2}(\tau,\xi)$, where
\begin{align*}
a_{j,1}(\tau,\xi)=&|\xi|^{-\varpi}\varphi_j(\xi)|\tau|^\beta
\int_{-\pi}^{\pi}
\^{\chi}_1(\tau-|\xi|\sin\theta)e^{-i\theta\varpi} d
\theta,\\
a_{j,2}(\tau,\xi)=&|\xi|^{-\varpi}\varphi_j(\xi)|\tau|^\beta
\int^\infty_0\^{\chi_2}\bigl(\tau-i|\xi|\sinh s\bigr)e^{-\varpi s}ds.\end{align*}
In what follows, we consider the contributions of $a_{j,1}$ and $a_{j,2}$, separately.\V

$\bullet$ \textbf{Estimation on the first error term}.
Write
\beq\label{eq:approx}
\mathscr F_{j,1}(f)(x,t)
=\iint e^{2\pi i(x\cdot\xi+t\tau)}\tilde a_{j,1}(\tau,\xi)\^f(\xi)d\tau d\xi\eeq
where
$\tilde a_{j,1}(\tau,\xi)
=\chi_{I_j}(\tau)a_{j,1}(\tau,\xi)$
and consider the error term
\begin{align*}
\mathscr R_{j,1}(f)(x,t)
:=\iint e^{2\pi i(x\cdot\xi+t\tau)}
a_{j,1}(\tau,\xi)\^f(\xi)d\tau d\xi-\mathscr F_{j,1}(f)(x,t).
\end{align*}
We have
$$\mathscr R_{j,1}(f)(x,t)=\int K_j(x,t,y)\tilde\D_j f(y) dy$$
where the distributional kernel is given by
$$
K_j(x,t,y)
=\int^{\pi}_{-\pi}\mathscr K_j(x,t,y;\theta)\,d\theta,
$$
and  $\mathscr K_j(x,t,y;\theta)$ may take the following form for any prefixed $N>0$
\begin{align*}
\mathscr K_j(x,t,y;\theta)
=e^{-i\theta\varpi}\int\bigl(1-\chi_{I_j}(\tau)\bigr)\mathscr H_j(\tau,\theta,N;x,y)\frac{|\tau|^\beta e^{2\pi it\tau}d\tau}{(1+|\tau|-2^{j+1})^{N}},
\end{align*}
provided that
$$
\mathscr H_j(\tau,\theta,N;x,y)=(1+|\tau|-2^{j+1})^N
\int\^{\chi_1}(\tau-|\xi|\sin \theta)\varphi_j(\xi)\frac{ e^{2\pi i(x-y)\cdot\xi}}{|\xi|^\varpi}d\xi.
$$

Next, from integration by parts,
we can seek a $C_N>0$ such that
\beq\label{eq:h-bdd}
|\mathscr H_j(\tau,\theta,N;x,y)|\leq C_N 2^{-j(\varpi-n)}\cdot 2^{100 j n}(1+2^j|x-y|)^{-100n}.
\eeq
Therefore, we have
\beq\label{eq:error-1}\|\mathscr R_{j,1}(f)(\cdot,t)\|_{L^p_x(\R^n)}\leq C'_N 2^{-j\frac N2}\|\tilde \D_j f\|_p,\;\text{when}\,\frac12\leq t\leq4,\eeq
for large $N$ and some $C'_N>0$.\V

$\bullet$ \textbf{Estimation on the second error term.}
 We will see
\beq\label{eq:approx}
\mathscr F_{j,2}(f)(x,t)
:=-\iint e^{2\pi i(x\cdot\xi+t\tau)}\tilde a_{j,2}(\tau,\xi)\^f(\xi)d\tau d\xi\eeq
carries the main contribution from $a_{j,2}$,
where
$\tilde a_{j,2}(\tau,\xi)
=\chi_{I_j}(\tau)a_{j,2}(\tau,\xi)$ and
$\chi_{I_j}(\tau)$ is defined the same as before.
Since the argument are very similar, we only sketch the proof below with some necessary remarks.\V

Similar to $\mathscr R_{j,1}$, there is an error term
$$
\mathscr R_{j,2}(f)(x,t)=\int_{\R^n} L_j(x,t,y)\tilde\D_j f(y)dy,
$$
with
\begin{align}\label{eq:L}
 L_j(x,t,y)=
\int^\infty_0ds&\int (1-\chi_{I_j}(\tau))|\tau|^\beta e^{2\pi i t\tau}e^{-\varpi s}d\tau\\
&\times\int_{\R^n}
\varphi_j(\xi)\^{\chi_2}\bigl(\tau-i|\xi|\sinh s\bigr)\frac{e^{2\pi i(x-y)\cdot \xi}}{|\xi|^\varpi}d\xi.
\nonumber
\end{align}
Noting that $\tau$ is restricted to $|\tau|> 2^{j+2}$,
we may insert $(1+|\tau|-2^{j+1})^{-N}$ into the integration and changing variables $\xi\rta 2^j\xi$.
Then, using integration by parts,
we can control
$$
 \^{\chi}_2\bigl(\tau-i 2^j|\xi|\sinh s\bigr)
=\int  \chi_2(r)e^{-2\pi(2^j|\xi|\sinh s)r}e^{-2\pi ir\tau}dr
$$
by
$$
\|\chi_2^{(N')}\|_{L^1}(1+|\tau|)^{-N'}e^{-\pi 2^{j-1}|\xi|\sinh s},\;\text{with}\;s>0,
$$
so that  $(1+|\tau|-2^{j+1})^N$  is absorbed by
$\^{\chi}_2(\tau-i 2^j|\xi|\sinh s)$.
Again, we also have
\beq\label{eq:error-2}
\|\mathscr R_{j,2}(f)(\cdot,t)\|_{L^p_x(\R^n)}
\leq C'_N 2^{-j\frac N2}\|\tilde \D_j f\|_p,\;
\text{when}\,\frac12\leq t\leq4\,,
\eeq
for large $N$ and some $C'_N>0$.\V

Since $\mathscr F_j (f)=\mathscr F_{j,1}(f)+\mathscr F_{j,2}(f)$, 
it remains to obtain the right estimate for the main term
$\mathscr F_j (f)$.\V

\textbf{Step 3.}
In this step, we evaluate the space-time
$L^p(\R^n\times\R)$ norm of
$\mathscr F_j(f)(x,t)$ by means of local smoothing.
Denote by $f_j(x)=f(2^{-j}x)$ and change variables $(\tau,\xi)\rta(2^j\tau,2^j\xi)$ to get
$$\mathscr F_{j}(f)(x,t)
=\iint e^{2\pi i 2^j(x\cdot\xi+t\tau)}2^j a_{j}(2^j\tau,2^j\xi)\^{f_j}(\xi)d\xi d\tau,$$
where
\beq\label{eq:changed-symbol}
2^j a_j(2^j\tau,2^j\xi)=2^{-j(\varpi-\beta)}\tilde{a}_{j}(\tau,\xi)
\eeq
$$
\tilde{ a}_j(\tau,\xi)=2^j\cdot2\pi|\xi|^{-\varpi}\varphi(\xi)|\tau|^\beta\chi_0(\tau)\int\tilde\chi(s)\mathcal J_{\varpi}(2\pi\cdot 2^js|\xi|)e^{-2\pi i2^j\tau s}ds.
$$
Using the asymptotic expansion of $\mathcal J_k(t)$ for $k>-\frac12$, we have
$$
\mathcal J_k(r)\simeq
r^{-\frac12}\Bigl[A_1(r)\cos\bigl(r-\frac{2k+1}{4}\pi\bigr)+A_2(r)\sin\bigl(r-\frac{2k+1}{4}\pi\bigr)
\Bigr]
$$
for large $r$ and
$$
A_1(r)=\sum^\infty_{\ell=0}c_1(\ell)r^{-2\ell},\;
A_2(r)=\sum^\infty_{\ell=0}c_2(\ell)r^{-2\ell-1},
$$
for some explicit coefficients $c_\sigma(\ell)$, $\sigma=1,2$.
As a consequence, we may write
$$
\tilde{a}_j(\tau,\xi)=\sqrt{2\pi}\,2^{-\frac
j2}\cdot2^j\tilde\chi_0(\tau)\sum_{\pm}\psi_{\varpi}^\pm(\xi)e^{\mp
i\frac{2\varpi+1}{4}\pi}\int \tilde{\tilde{\chi}}(s)\cdot e^{-2\pi i
2^j s(\tau\mp |\xi|)} ds\,,
$$
where
$\psi_{\varpi}^\pm(\xi)=\varphi(\xi)\mathcal A_\pm(2\pi 2^j\xi)|\xi|^{-\varpi-\frac12}$
and
\begin{equation}
\left\{ \begin{aligned}
         \tilde\chi_0(\tau)=&|\tau|^\beta\chi_0(\tau)\\
                \tilde{\tilde{\chi}}(s)=&\tilde\chi(s)s^{-\frac12}\,.
                          \end{aligned} \right.
                          \end{equation}
Here, $\mathcal A_{\pm}(\xi)=(A_1(|\xi|)\mp iA_2(|\xi|))/2$ belongs to
the classical symbol of order zero $S^0_{1,0}(\R^n)$.\V

Since the function is localized in the
high frequency, we now involve the asymptotic expansion for
Bessel functions which leads to the half wave operator.
Write
\beq\label{eq:F}
\mathscr F_{j}(f)(x,t)=\mathscr F^+_{j}(f)(x,t)+\mathscr F^-_{j}(f)(x,t),
\eeq
where
\begin{align*}
\mathscr F^\pm_{j}(f)(x,t)
&=2^{-j(\varpi-\beta+\frac12)}e^{\mp i\frac{2\varpi+1}{4}\pi}
\iint e^{2\pi i 2^j(x\cdot\xi+t\tau)}\mathscr A_{\pm}(\tau,\xi)\^{f_j}(\xi)d\xi d\tau,\\
\mathscr A_\pm(\tau,\xi)
&=2^j\tilde\chi_0(\tau)\psi_{\varpi}^\pm(\xi)\int \tilde{\tilde{\chi}}(s)\cdot e^{-2\pi i 2^j s(\tau\mp |\xi|)}ds.
\end{align*}
It is not hard to see
\begin{align*}
\mathscr F^\pm_{j}(f)(x,t)= &2^{-j(\varpi-\beta+\frac12)}e^{\mp i\frac{2\varpi+1}{4}\pi}2^j\\
&\times\iint e^{2\pi i 2^jx\cdot\xi}\psi_{\varpi}^\pm(\xi)\^{\tilde\chi_0}(2^j(s-t))
\tilde{\tilde{\chi}}(s)\^{f_j}(\xi)e^{\pm2\pi i2^j|\xi|s}d\xi ds\\
=&2^{-j(\varpi-\beta+\frac12)}e^{\mp i\frac{2\varpi+1}{4}\pi}2^j\\
&\times\iint e^{2\pi i x\cdot\xi}\psi_{\varpi}^\pm(2^{-j}\xi)\^{\tilde\chi_0}(2^j(s-t))
\tilde{\tilde{\chi}}(s)\^{f}(\xi)e^{\pm2\pi i|\xi|s}d\xi ds\\
=&2^{-j(\varpi-\beta+\frac12)}e^{\mp i\frac{2\varpi+1}{4}\pi}2^j\\
&\times\iint\tilde{\tilde{\chi}}(s)e^{\pm
is\sqrt{-\D}}\tilde\D_jf(y)\^{\tilde\chi_0}(2^j(s-t))\check{\psi}^\pm_\varpi(2^j(x-y))2^{jn}dyds.
\end{align*}
%
Applying Young's inequality to the above expressions, we have
$$\|\mathscr F^\pm_{j}(f)\|_{L^p(\R^n\times\R)}\leq 2^{-j(\varpi-\beta+\frac12)}\|\check{\psi}^\pm_\varpi\|_{L^1(\R^n)}
\|\^{\tilde\chi_0}\|_{L^1(\R)}\|\tilde{\tilde{\chi}}(s)u^\pm_j(s,y)\|_{L^p(\R^n\times
\R)},$$ where $u^\pm_j(s,y):=e^{\pm is\sqrt{-\D}}\tilde \D_jf(y)$.\V

 At this stage, we use local smoothing estimate for wave equations with $p>2$ to get
$$
\|\mathscr F^\pm_{j}(f)\|_{L^p(\R^n\times\R)}\leq C 2^{j\mu}\|f\|_{L^p(\R^n)},
$$
where  $\mu=-\varpi+\beta-\frac12+\gamma,\;\forall \,\gamma>\gamma(p,n)$.
In view of \eqref{eq:F},
we have under the same conditions
$$
\|\mathscr F_{j}(f)\|_{L^p(\R^n\times\R)}\leq C 2^{j\mu}\|f\|_{L^p(\R^n)}.
$$
Summing up all the estimates on $\mathscr F_j(f)(x,t)$ by means of Littlewood-Pelay's theory,
we have in view of \eqref{eq:error-1} and \eqref{eq:error-2}
$$\Bigl\||\p_t|^\beta\bigl(\chi(t)\mathscr M^\a_t(f)\bigr)\Bigr\|_{L^p(\R^n\times \R)}\leq C\sum_{j\geq 0}2^{\mu j}\|f\|_p+ C'_N 2^{-\frac N2}\| f\|_p,$$
for some suitable constant $C>0$. To ensure the geometric series converge, we need
$$\mu<0,\;\gamma>\gamma(p,n),$$
which in turn is equivalent to
$$\beta<\a+\frac{n-1}{2}-\gamma(p,n).$$
On the other hand, if $\beta>\frac1p$, we can use Sobolev embedding to obtain
\begin{align}\label{eq:truncated}
\Bigl\| \sup_{1<t<2}|\mathscr M^\a_t (f)(x)|&\Bigr\|_{L^p(\R^n)}\\
\nonumber\leq& C\Bigl\|\langle\p_t\rangle^\beta\bigl(\chi(t)\mathscr
M^\a_t(f)\bigr)\Bigr\|_{L^p(\R^n\times \R)} \leq C\|f\|_p,
\end{align}
where $\langle\p_t\rangle=(1-\p^2_t)^\frac12$.
The conditions for the exponents $\mu,\a,\beta$ are summarized to
\beq\label{eq:2.11*}
\a>-\varepsilon(p,n).
\eeq

\textbf{Step 4.} In this step, we reduce the general situation when $t>0$
to the particular case of $t\in[1,2]$. To achieve this, we will use Littlewood-Paley theory by writing
$$\mathscr M^\a_tf(x)=\sum^\infty_{j=0}\mathscr M^\a_{j,\,t}f(x),$$
where
$$\mathscr M^\a_{j,\,t}f(x)=\int e^{2\pi ix\cdot\xi}\varphi_j(t\xi)\^m_\a(t\xi)\^f(\xi)d\xi,$$
with $\varphi_j(\xi)$ defined at the beginning of the proof.
Thus it suffices to see for $\a$ and $p$ satisfying \eqref{eq:2.11*},
there exists some $\mu<0$ such that
\beq\label{eq:scaled-reduce-1}
\Bigl\|\sup_{t>0}|\mathscr M^\a_{j,\,t}f(x)|\Bigr\|_{L^p(\R^n)}\leq C\cdot 2^{j\mu}\|f\|_{L^p}.
\eeq
If $j>0$, we may now reduce \eqref{eq:scaled-reduce-1} to \eqref{eq:truncated}. Denote by $\^{\dot\D_\ell f}(\xi)=\varphi(2^{-\ell}|\xi|)\^f(\xi)$ for all $\ell\in \Z$ and notice that we have
$$
\mathscr M^\a_{j,\,t}f(x)=\sum_{|\ell|\leq 100}\mathscr M^\a_{j,\,t}(\dot\D_{j+k+\ell}f)(x)\,,
$$
whenever $2^{-k}\leq t\leq 2^{-k+1}$ for $k\in \Z$, and the following pointwise estimation
$$
\sup_{t>0}|\mathscr M^\a_{j,\,t}f(x)|\leq\Bigl( \sum_{k\in \Z}\sup_{2^{-k}\leq t\leq 2^{-k+1}}|\mathscr M^\a_{j,\,t}f(x)|^p\Bigr)^\frac1p.
$$
Now, we claim that
\beq\label{eq:scaled-reduce-2}
 \Bigl\|\sup_{t\in
[2^{-k},2^{-k+1}]}|\mathscr M^\a_{j,\,t}f(x)|\Bigr\|_{p} \leq C\cdot
2^{\mu j}\|f\|_p,\;\forall \,k\in \Z\,.
\eeq
Under this claim, the left side of
\eqref{eq:scaled-reduce-1} can be estimated as
\begin{align*}
&\Bigl(\sum_{k\in \Z}\Bigl\|\sup_{t\in [2^{-k},2^{-k+1}]}\mathscr M^\a_{j,\,t}\Bigl(\sum_{|\ell|\leq 100}\dot\D_{j+k+\ell}f\Bigr)(x)\Bigr\|^p_{p}\Bigr)^\frac1p\\
\leq &C\cdot2^{\mu j}\Bigl(\sum_{k\in \Z}\Bigl\|\sum_{|\ell|\leq 100}\dot\D_{j+k+\ell}f(x)\Bigr\|^p_{p}\Bigr)^\frac1p\\
\leq &C'\cdot2^{\mu j}\Bigl\|\Bigl( \sum_{k\in
\Z}|\dot\D_{k}f(x)|^2\Bigr)^\frac12\Bigr\|_{p}.\end{align*} Invoking
the standard square-function inequality, we see
 the last term is clearly bounded by $2^{\mu j}\|f\|_p$, where $2<p<\infty$.\V

To show \eqref{eq:scaled-reduce-2}, we start with a standard scaling
to get \beq\label{eq:scaled-reduce-3} \Bigl\|\sup_{t\in
[2^{-k},2^{-k+1}]}\big|\mathscr M^\a_{j,\,t}f(x)\big|\Bigr\|_{p} \leq
\Bigl\|\sup_{t\in [1,2]}\big|\mathscr
M^\a_{j,t}(f_{2^k})(2^kx)\big|\Bigr\|_{p}, \eeq where
$f_{2^k}(x)=f(2^{-k}x)$.
In view of the following formula,
$$
\^{\mathscr M^\a_{j,\,t} f}(\xi)=\^{m}_\a(t\xi)\varphi_j(\xi)\^f(\xi)+\int^t_0\^m_\a(t\xi)
\langle\xi,\nabla\varphi_j(\theta\xi)\rangle\^f(\xi)d\theta,
$$
we have
$$
\sup_{1<t<2}|\mathscr M^\a_{j,\,t} f(x)|\leq\sup_{t\in [1,2]}|\mathscr M^\a_t(\D_jf)(x)|+\int^2_1\sup_{t\in[1,2]}
|\mathscr M^\a_t(\tilde\D^\theta_jf)(x)|d\theta,
$$
where $$\^{\tilde\D^\theta_j f}(\xi):=\langle\xi,\nabla\varphi_j(\theta\xi)\rangle\^f(\xi).$$
The argument in Step 3 works to $\sup_{t\in[1,2]}
|\mathscr M^\a_t(\tilde\D^\theta_jf)(x)|$ as well, yielding an appropriate upper bound independent of $\theta\in [1,2]$.
Therefore \eqref{eq:scaled-reduce-2} follows from Step 3 and rescaling.\V

It remains to handle the case when $j=0$. This is standard as we denote
$\^\Phi(\xi)=\varphi_0(\xi)\^m_\a(\xi)$, then $\int\Phi(x)dx=\^m_\a(0)$ and
$$
\mathscr M_{0,\,t}^\a f(x,t)=\Phi_t*f(x),
$$
where $\Phi_t(x)=t^{-n}\Phi(x t^{-1})$. Hence
$$
\sup_{t>0}\bigl|\mathscr M_{0,\,t}^\a f(x,t)\bigr|\leq C M_{HL}(f)(x),
$$
and the $L^p$ estimate follows.\V
%

Finally, notice that the above arguments also work when $\a\in \mathbb C$ and ${\rm Re}\;\a>-\varepsilon(p,n)$.
We complete the proof of Theorem \ref{thm:main} and this extends Stein's result for $n\geq 2$ and $p> 2$.
\end{proof}
\begin{remark}
If one  asks the same question for general hypersurface rather than standard sphere, for example, we replace $|x|=\sqrt{x^2_1+\cdots x^2_n}$ by another norm
$$\|x\|_s=(x^s_1+\cdots+x^s_n)^{\frac1s}, 0<s<\infty,$$ in the definition of $m_\a(x)$ with $\a\geq 0$,  the above results deduced above
fails to hold.
In fact, it is shown in \cite{IosevichSawyer} that if one considers the surfaces where the Gaussian curvature
is allowed to vanish, the $L^{p}$ exponents for the corresponding  maximal operators are often worse.
We refer to  \cite{IosevichSawyer} for those cases.
\end{remark}

\section{Further discussions}

At the end of this paper, we  discuss on some directions
which might be helpful to further studies.
It seems interesting to study the following problems.\V

(1). Is the relation ${\rm Re}\, \a>-\varepsilon(p,n)$ implied by local smoothing estimate optimal for \eqref{eq:stein} to hold ?
\V

(2). The researches concerning Stein's maximal spherical operators appear in the literature focusing mainly on two aspects.
One is to study the variable coefficient version of the maximal functions when $\a=0$, as can be found in \cite{Sogge1993}.
The other one is to extend the relation between $\a$ and $p$ in Theorem \ref{thm:main} for classical maximal operator \eqref{eq:def-max-spherical}.
A natural question is how to combine these two directions by establishing the variable coefficient version of the analytic family of spherical means.\V

(3). It is well known that Stein's $L^p-L^p$ bounds on maximal spherical means can be generalized to certain $L^q-L^p$ inequalities, see Schlag \cite{Schlag1996,Schlag1997}.
If $\a=0$, there is a variable coefficient version of Schlag's $L^q-L^p$ estimates in \cite{Schlag-Sogge1997}.
The question is whether this is valid for $\a\neq 0$.\V

(4). When $\a=0$, it is shown that $p>\frac n{n-1}$ is necessary for \eqref{eq:stein}. It is interesting to know if the weak type $(p,p)$ estimate holds at the end-point $p=\frac{n}{n-1}$. At this stage, classical Calder\'on-Zygmund decomposition may be useful.\V


(5). It seems also interesting to know whether it is possible to prove the maximal inequality for extended exponents in Theorem \ref{thm:main} without using local smoothing estimate. If this is true, it will be an evidence to support the likely true local smoothing conjecture.
\V



\appendix
\section{}
For the convenience of reading, we give the proof of two
facts cited in Section 1.
\subsection{The proof of \eqref{eq:FT-of-m}}
Using polar
coordinates, we can do the following calculation
\begin{align*}
\^m_{\a}(\xi) =&\int_{\R^n}e^{-2\pi x\cdot\xi}m_{\a}(x)dx\\
=&\frac1{\Gamma(\a)}\int_{\R^n}e^{-2\pi
x\cdot\xi}(1-|x|^2)^{\a-1}_+dx\\
=&\frac1{\Gamma(\a)}\int_0^1(1-\rho^2)^{\a-1}\rho^{n-1}\int_{\mathbb
S^{n-1}}e^{-2\pi \rho w\cdot \xi}d\sigma(w) d\rho\\
=&\frac{2\pi}{\Gamma(\a)}\int_0^1(1-\rho^2)^{\a-1}\rho^{\frac{n-2}2+1}J_{\frac{n-2}2}(2\pi|\xi|\rho)d\rho\\
=&\pi^{-\a+1}|\xi|^{-\frac n2-\a+1}\mathcal J_{\frac
n2+\a-1}(2\pi|\xi|)\,,
\end{align*}
where we have used the following two identities (see Appendix B in
\cite{Grafakos})
\begin{align*}
\int_{\mathbb S^{n-1}}e^{-2\pi \theta\cdot
\xi}d\theta=&\frac{2\pi}{|\xi|^{\frac{n-2}2}}J_{\frac{n-2}2}(2\pi|\xi|)\,,\\
\int_0^1J_{\mu}(ts)s^{\mu+1}(1-s^2)^\nu
ds=&\frac{\Gamma(\nu+1)2^\nu}{t^{\nu+1}}J_{\mu+\nu+1}(t)\,.
\end{align*}

\subsection{The proof of \eqref{eq:sph-wave}}
Here we give a new proof of \eqref{eq:sph-wave}.
It is based on the following fact
$$J_{\frac12}(r)=\frac{\sqrt{r}}{\sqrt{2}\Gamma(1/2)}\int_{-1}^1e^{isr}ds=\frac1{\sqrt{2}\Gamma(1/2)}\frac{\sin(r)}{\sqrt{r}}\,.$$
Letting $\a=\frac{3-n}2$, we obtain
\begin{align*}
t \mathcal{F}_{x\rta \xi}\big(M^\a_t(f)(\cdot)\big)(\xi)=&t\hat
m_\a(t\xi)\hat{f}(\xi)\\
=&t \frac{\pi^{\frac{n-1}2}}{|t\xi|^\frac12}J_\frac12(2\pi|t\xi|)\^f(\xi)\\
=&\tilde c_n\frac{\sin(t\xi)}{|\xi|}\hat{f}(\xi)\\
=&\tilde
c_n\mathcal{F}_{x\rta \xi}\Big(\frac{\sin(t\sqrt{-\Delta})}{\sqrt{-\Delta}}f\Big)(\xi)\,.
\end{align*}
After taking inverse Fourier transform, we have \eqref{eq:sph-wave} solves the Cauchy problem \eqref{eq:1}.

\end{CJK*}
\end{document}